\documentclass[12pt, 14paper]{amsart}
\usepackage{amsmath,amsfonts,amssymb}

\usepackage{longtable}
\usepackage[mathscr]{eucal}
\usepackage{amsmath, amsthm}
\usepackage{mathrsfs}
\usepackage{amsbsy}
\usepackage{wasysym}
\usepackage{url}

\input xypic
\xyoption{all}

\makeindex
\makeglossary

\begin{document}
\baselineskip = 16pt

\newcommand \ZZ {{\mathbb Z}}
\newcommand \NN {{\mathbb N}}
\newcommand \RR {{\mathbb R}}
\newcommand \PR {{\mathbb P}}
\newcommand \AF {{\mathbb A}}
\newcommand \GG {{\mathbb G}}
\newcommand \QQ {{\mathbb Q}}
\newcommand \CC {{\mathbb C}}
\newcommand \bcA {{\mathscr A}}
\newcommand \bcC {{\mathscr C}}
\newcommand \bcD {{\mathscr D}}
\newcommand \bcF {{\mathscr F}}
\newcommand \bcG {{\mathscr G}}
\newcommand \bcH {{\mathscr H}}
\newcommand \bcK{{\mathscr K}}
\newcommand \bcM {{\mathscr M}}
\newcommand \bcJ {{\mathscr J}}
\newcommand \bcL {{\mathscr L}}
\newcommand \bcO {{\mathscr O}}
\newcommand \bcP {{\mathscr P}}
\newcommand \bcQ {{\mathscr Q}}
\newcommand \bcR {{\mathscr R}}
\newcommand \bcS {{\mathscr S}}
\newcommand \bcV {{\mathscr V}}
\newcommand \bcW {{\mathscr W}}
\newcommand \bcX {{\mathscr X}}
\newcommand \bcY {{\mathscr Y}}
\newcommand \bcZ {{\mathscr Z}}
\newcommand \goa {{\mathfrak a}}
\newcommand \gob {{\mathfrak b}}
\newcommand \goc {{\mathfrak c}}
\newcommand \gom {{\mathfrak m}}
\newcommand \gon {{\mathfrak n}}
\newcommand \gop {{\mathfrak p}}
\newcommand \goq {{\mathfrak q}}
\newcommand \goQ {{\mathfrak Q}}
\newcommand \goP {{\mathfrak P}}
\newcommand \goM {{\mathfrak M}}
\newcommand \goN {{\mathfrak N}}
\newcommand \uno {{\mathbbm 1}}
\newcommand \Le {{\mathbbm L}}
\newcommand \Spec {{\rm {Spec}}}
\newcommand \Gr {{\rm {Gr}}}
\newcommand \Pic {{\rm {Pic}}}
\newcommand \Jac {{{J}}}
\newcommand \Alb {{\rm {Alb}}}
\newcommand\alb{{\rm{alb}}}
\newcommand \Corr {{Corr}}
\newcommand \Chow {{\mathscr C}}
\newcommand \Sym {{\rm {Sym}}}
\newcommand \Prym {{\rm {Prym}}}
\newcommand \cha {{\rm {char}}}
\newcommand \eff {{\rm {eff}}}
\newcommand \tr {{\rm {tr}}}
\newcommand \Tr {{\rm {Tr}}}
\newcommand \pr {{\rm {pr}}}
\newcommand \ev {{\it {ev}}}
\newcommand \cl {{\rm {cl}}}
\newcommand \interior {{\rm {Int}}}
\newcommand \sep {{\rm {sep}}}
\newcommand \td {{\rm {tdeg}}}
\newcommand \alg {{\rm {alg}}}
\newcommand \im {{\rm im}}
\newcommand \gr {{\rm {gr}}}
\newcommand \op {{\rm op}}
\newcommand \Hom {{\rm Hom}}
\newcommand \Hilb {{\rm Hilb}}
\newcommand \Sch {{\mathscr S\! }{\it ch}}
\newcommand \cHilb {{\mathscr H\! }{\it ilb}}
\newcommand \cHom {{\mathscr H\! }{\it om}}
\newcommand \colim {{{\rm colim}\, }} % colimit
\newcommand \End {{\rm {End}}}
\newcommand \coker {{\rm {coker}}}
\newcommand \id {{\rm {id}}}
\newcommand \van {{\rm {van}}}
\newcommand \spc {{\rm {sp}}}
\newcommand \Ob {{\rm Ob}}
\newcommand \Aut {{\rm Aut}}
\newcommand \cor {{\rm {cor}}}
\newcommand \Cor {{\it {Corr}}}
\newcommand \res {{\rm {res}}}
\newcommand \red {{\rm{red}}}
\newcommand \Gal {{\rm {Gal}}}
\newcommand \PGL {{\rm {PGL}}}
\newcommand \Bl {{\rm {Bl}}}
\newcommand \Sing {{\rm {Sing}}}
\newcommand \spn {{\rm {span}}}
\newcommand \Nm {{\rm {Nm}}}
\newcommand \inv {{\rm {inv}}}
\newcommand \codim {{\rm {codim}}}
\newcommand \Div{{\rm{Div}}}
\newcommand \CH{{\rm{CH}}}
\newcommand \sg {{\Sigma }}
\newcommand \DM {{\sf DM}}
\newcommand \Gm {{{\mathbb G}_{\rm m}}}
\newcommand \tame {\rm {tame }}
\newcommand \znak {{\natural }}
\newcommand \lra {\longrightarrow}
\newcommand \hra {\hookrightarrow}
\newcommand \rra {\rightrightarrows}
\newcommand \ord {{\rm {ord}}}
\newcommand \Rat {{\mathscr Rat}}
\newcommand \rd {{\rm {red}}}
\newcommand \bSpec {{\bf {Spec}}}
\newcommand \Proj {{\rm {Proj}}}
\newcommand \pdiv {{\rm {div}}}
\newcommand \wt {\widetilde }
\newcommand \ac {\acute }
\newcommand \ch {\check }
\newcommand \ol {\overline }
\newcommand \Th {\Theta}
\newcommand \cAb {{\mathscr A\! }{\it b}}

\newenvironment{pf}{\par\noindent{\em Proof}.}{\hfill\framebox(6,6)
\par\medskip}

\newtheorem{theorem}[subsection]{Theorem}
\newtheorem{conjecture}[subsection]{Conjecture}
\newtheorem{proposition}[subsection]{Proposition}
\newtheorem{lemma}[subsection]{Lemma}
\newtheorem{remark}[subsection]{Remark}
\newtheorem{remarks}[subsection]{Remarks}
\newtheorem{definition}[subsection]{Definition}
\newtheorem{corollary}[subsection]{Corollary}
\newtheorem{example}[subsection]{Example}
\newtheorem{examples}[subsection]{examples}

\title{On finite dimensionality of Chow groups}
\author{Kalyan Banerjee}

\email{banerjeekalyan@hri.res.in}
\begin{abstract}
In this exposition we understand when the natural map from the two-fold self product of the Chow variety parametrizing codimension $p$ cycles on a smooth projective variety $X$ to the Chow group $\CH^p(X)_0$ of degree zero cycles is surjective. We derive some consequences when the map is surjective.
\end{abstract}
\maketitle
%\tableofcontents

\section{Introduction}
The representability question in the theory of Chow groups is an important question. Precisely it means the following: let $X$ be a smooth projective variety over the field of complex numbers and let $\CH^p(X)$ denote the Chow group of codimension $p$ algebraic cycles on $X$ modulo rational equivalence. Let $\CH^p(X)_{\alg}$ denote the subgroup of $\CH^p(X)$ consisting of algebraically trivial cycles. We say that the group $\CH^p(X)_{\alg}$ is representable if there exists a smooth projective curve $C$ and a correspondence $\Gamma$ on $C\times X$ such that $\Gamma_*$ from $J(C)\cong \CH^1(C)_{\alg}$ to $\CH^p(X)_{\alg}$ is surjective.  The most interesting and intriguing is the case of the zero cycles on $X$. The first breakthrough result in this direction is  by Mumford, \cite{M}: the group $\CH^2(S)_{\alg}$ of a smooth projective complex algebraic surface with geometric genus greater than zero is not representable. It was further generalised by Roitman \cite{R1} for higher dimensional varieties proving that if the variety $X$ has a global holomorphic $i$-form on it, then the group $\CH^n(X)_{\alg}$ is not representable. Here $n$ is the dimension of $X$ and we have $0<i\leq n$. Then there is the famous converse question due to Spencer Bloch saying that: for a smooth projective complex algebraic surface $S$ with geometric genus equal to zero, the group $\CH^2(S)_{\alg}$ is representable. This question has been answered for the surfaces not of general type with geometric genus equal to zero by Bloch-Kas-Liebarman, in \cite{BKL}. The conjecture is still open in general for surfaces of general type but it has been solved in many examples in \cite{B}, \cite{IM}, \cite{GP}, \cite{PW}, \cite{V}, \cite{VC}.

In  the case of highest codimensional cycles on a smooth projective variety $X$, the notion of representability can also be defined in another way. We consider the natural map from the symmetric power $\Sym^d X$ to $\CH^n(X)_{\alg}$, for $d$ positive and $n$ being the dimension of $X$. Suppose that this map is surjective for some $d$ then we say that the group $\CH^n(X)_{\alg}$ is representable. It can be proved as in [\cite{Vo}, Chapter 10, Section 1], that this notion of representability is equivalent to the first notion of representability introduced before.

So following the approach of Voisin as in \cite{Vo}, it is natural to ask whether there is a second notion of representability for lower codimensional cycles. To do this, first we assume that there exist a smooth, projective variety $Y$ and a correspondence $\Gamma$ of codimension $p$ on $Y\times X$, such that $\Gamma_*$ is surjective from ${\CH_0(Y)}_{0}\cong {\CH_0(Y)}_{\alg}$ to $\CH^p(X)_0$, (here ${\CH_0(Y)}_0,\CH^p(X)_0$ denote the group of degree zero cycles on $Y,X$ respectively modulo rational equivalence). Under this assumption the group $\CH^p(X)_0$ is contained in $\CH^p(X)_{\alg}$. On the other hand $\CH^p(X)_{\alg}$ is always contained in $\CH^p(X)_0$ (this fact is due to the existence of Chow varieties, for reference please see [\cite{Fulton}, Example 10.3.3]) Now the second notion of representability is as follows: Let us consider the two-fold product of the Chow variety $C^p_{d}(X)\times C^p_{d}(X)$ where $C^p_d(X)$ is the projective variety parametrizing codimension $p$ and degree $d$ cycles on $X$. Note that the degree of a cycle is determined by fixing an embedding of $X$ into a projective space. Consider the natural map from this product to the group $\CH^p(X)_{0}$ and hence (under the assumption) to $\CH^p(X)_{\alg}$ and we ask the question: does the surjectivity of this map implies the representability in the first sense of $\CH^p(X)_{\alg}$. First we prove the following in this direction:

\smallskip

\begin{theorem}Suppose that the Chow group of codimension $p$ cycles is generated by linear subspaces, that is the natural (Abel-Jacobi) map from $\CH_0(F(X))$ to $\CH^p(X)$ is surjective. This condition is enough to ensure the above mentioned condition on the subgroup $\CH^p(X)_0$ of $\CH^p(X)$. Here $F(X)$ is the Fano variety of linear subspaces of codimension $p$. Suppose that $\CH^p(X)_{\alg}$ is representable in the sense that the map from the two-fold product of the Chow variety to $\CH^p(X)_{\alg}$ is surjective. Then there exists a smooth projective curve $C$ and a correspondence $\Gamma$ on $C\times X$, such that
$\Gamma_*:\CH^1(C)_{\alg}\to \CH^p(X)_{\alg}$
is surjective.
\end{theorem}

\smallskip

As an application we show that the natural map from $C^3_{d}(X)\times C^3_d(X)$ to $\CH^3(X)_{\alg}$ is not surjective for any $d$, where $X$ is a cubic fourfold embedded in $\PR^5$.

Our argument in this direction is a minor modification of the argument present in the approach of Roitman in \cite{R},  Voisin in [\cite{Vo}, Chapter 10, Section 1], where the authors deal with the case of zero cycles. First we recall various notions of representability in the second sense, denoted as "finite dimensionality" of Chow groups of codimension $p$ cycles and show their equivalence. The key point is to use the Roitman's result on the map from the two-fold product of the Chow variety to $\CH^p(X)_{\alg}$ saying that the fibers of this map is a countable union of Zariski closed subsets in the product of Chow varieties. Then after having these equivalent notions of "finite dimensionality" in hand we proceed to the main theorem.

\smallskip

{\small \textbf{Acknowledgements:} The author would like to thank the hospitality of Tata Institute Mumbai and Harish Chandra Reserach Institute, India, for hosting this project. The author also thanks  the anonymous referee for careful reading of the manuscript and for advising the author about the improvements of certain technical points in the paper.}

{\small Throughout the text we work over the field of complex numbers.}

\section{Finite dimensionality of the Chow groups}
Let $X$ be a smooth projective variety defined over the ground field $k$. Let $C^p_d(X)$ denote the Chow variety of $X$ parametrizing codimension $p$ effective cycles on $X$ of a certain degree $d$. To consider the degree we fix an embedding of $X$ into some projective space $\PR^N$. Consider the $\CC$-points of the variety $C^p_{d}(X)$ and denote the set of $\CC$-points of $C^p_d(X)$, simply as $C^p_d(X)$. Then consider the map
$$\theta_p^d:C^p_d(X)\times C^p_d(X)\to \CH^p(X)_{0}$$
given by
$$(Z_1,Z_2)\mapsto [Z_1-Z_2]$$
where $[Z_1-Z_2]$ is the class of the cycle $Z_1-Z_2$ in the Chow group. By abusing notation we will denote the class $[Z_1-Z_2]$ as $Z_1-Z_2$. We make the following assumption:

(*) There exist a smooth, projective variety $Y$ and a correspondence $\Gamma$ of codimension $p$ on $Y\times X$, such that $\Gamma_*$ is surjective from ${\CH_0(Y)}_{0}\cong {\CH_0(Y)}_{\alg}$ to $\CH^p(X)_0$.

Under this assumption the map $\theta^p_d$ takes value in the group $\CH^p(X)_{\alg}$, as $\CH^p(X)_0$ is contained in $\CH^p(X)_{\alg}$ and $\CH^p(X)_{\alg}$ is always contained in $\CH^p(X)_0$ by the theory of Chow varieties, for reference please see [\cite{Fulton}, Example 10.3.3].

\begin{definition}
\label{definition1}
We say that the group $\CH^p(X)_{\alg}$ is representable if there exists $d$ such that $\theta^p_d$ is surjective.
\end{definition}
Now the natural question is that, what are the fibers of $\theta^p_d$, for a fixed $p,d$. Here is a theorem about that
\begin{theorem}
\label{theorem1}
Each fiber of  the map $\theta^p_d$ is a countable union of Zariski closed subsets of $C^p_d(X)\times C^p_d(X)$.
\end{theorem}
\begin{proof}
The proof of this theorem follows closely the approach present in [\cite{R}, Theorem 1], [\cite{Vo}, Lemma 10.7]. This theorem is also proved in [\cite{Ba}, Proposition 4.1.1] and [\cite{BG}, Proposition 8], but for the sake of completeness of the arguments and for the importance of this theorem in this paper, we include a proof of it.

To prove this we consider the following reformulation of the definition of rational equivalence. Let $Z_1,Z_2$ be two codimension $p$ cycles. They are rationally equivalent if there exists an effective cycle $Z'$, such that $Z_1+Z',Z_2+Z'$ belong to $C^p_d(X)$ for some fixed $d$, and there exists a regular morphism $f$ from $\PR^1$ to $C^p_d(X)$, such that
$$f(0)=Z_1+Z',\quad f(\infty)=Z_2+Z'\;.$$

Let us consider two cycles $Z_1,Z_2$ belonging to ${\theta^p_d}^{-1}(z)$ for some rational equivalence class $z$. Then $Z_1,Z_2$ belong to ${\theta^p_d}^{-1}(z)$ means that $Z_1,Z_2$ are rationally equivalent. So there exists $Z',f$ as above such that
$$f(0)=Z_1+Z',\quad f(\infty)=Z_2+Z'\;.$$
Thus it is natural to consider the following subvarieties of $C^p_d(X)\times C^p_d(X)$ denoted by $W^{u,v}_d$, given as the collection of all $(Z_1,Z_2)$ so that there exist $Z'\in C^p_{u}(X)$ $f$ in $\Hom^v(\PR^1,C^p_{d+u}(X))$, for some positive integer $u$ satisfying
$$f(0)=Z_1+Z',\quad f(\infty)=Z_2+Z'\;.$$
Here $\Hom^v(\PR^1,C^p_{d+u}(X))$ is the Hom-scheme of degree $v$ morphisms from $\PR^1$ to $C^p_{d+u}(X)$.
For working purpose denote $\prod_{i=1}^n C^p_{d_i}(X)$ as $C^p_{d_1,\cdots,d_n}(X)$.

  Let
  $$
  e:\Hom ^v(\PR ^1,C^p_{d+u,d+u}(X))\to
  C^p_{d+u,d+u}(X)
  $$
be the evaluation morphism sending $f:\PR ^1\to C^p_{d+u,d+u}(X)$ to the ordered pair $(f(0),f(\infty ))$, and let us consider the diagonal in $C^p_{u,u}(X)$ and multiply with $C^p_{d,d}(X)$, call it $F$ and let
  $$
  s:F\to C^p_{d+u,d+u}(X)
  $$
be the regular morphism sending $(Z_1,Z_2)$ to $(Z_1+Z',Z_2+Z')$. The two morphisms $e$ and $s$ allow to consider the fibred product
  $$
  V=\Hom ^v(\PR ^1,C^p_{d+u,d+u}(X))\times _{C^p_{d+u,d+u}(X)}F\; .
  $$
This $V$ is a closed subvariety in the product
  $$
  \Hom ^v(\PR ^1,C^p_{d+u,d+u}(X))\times F
  $$
over $\Spec (k)$ consisting of quintuples $(f,Z_1,Z_2,Z')$ such that
  $$
  e(f)=s(Z_,Z_2,Z')\; ,
  $$
i.e.
  $$
  (f(0),f(\infty ))=(Z_1+Z',Z_2+Z')\; .
  $$
The latter equality gives
  $$
  \pr(V)\subset W_d^{u,v}\; .
  $$
where $\pr$ is the projection from $V$ to $C^p_d(X)\times C^p_d(X)$.

Vice versa, if $(Z_1,Z_2)$ is a closed point of $W_d^{u,v}$, there exists a regular morphism
  $$
  f\in \Hom ^v(\PR ^1,C^p_{d+u,d+u}(X))
  $$
with $ f(0)=Z_1+Z'$ and $ f(\infty )=Z_2+Z'$. Then $(f,Z_1,Z_2,Z')$ belongs to $V$.

So the set $W_d^{u,v}$ is equal to $\pr(V)$. Since $V$ itself a quasi-projective variety, $W_d^{u,v}$ is a constructible subset in the product $C^p_d(X)\times C^p_d(X)$.

Suppose that $(Z_1,Z_2)$ is in $W_d^{u,v}$. Then there exists $f$ in $\Hom^v(\PR^1,C^p_{d+u}(X))$, $Z'$ in $C^p_u(X)$ such that

$$f(0)=Z_1+Z',\quad f(\infty)=Z_2+Z'\;.$$
This immediately imply that $(Z_1+Z',Z_2+Z')$ is in $W_{d+u}^{0,v}$. On the other hand, consider the map
$$\tilde{s}:C^p_{d,d}(X)\times \Delta_{C^p_{u,u}(X)}\to C^p_{d+u,d+u}(X)$$
given by
$$(Z_1,Z_2,Z')\mapsto (Z_1+Z',Z_2+Z')\;.$$
By the above we have that
$$W_d^{u,v}\subset \pr_{1,2}(\tilde{s}^{-1}(W_{d+u}^{0,v}))\;.$$
Conversely suppose that $(Z'_1,Z'_2)$ belongs to $\pr_{1,2}(\tilde{s}^{-1}(W_{d+u}^{0,v}))$. Then $(Z'_1,Z'_2)$ is of the form $(Z_1+Z',Z_2+Z')$, such that there exists $f$ in $\Hom^v(\PR^1,C^p_{d+u}(X))$ satisfying
$$f(0)=Z_1+Z',\quad f(\infty)=Z_2+Z'\;.$$
This tells us that $(Z_1,Z_2)$ belongs to $W_d^{u,v}$.
Hence we have that
  $$
  W_d^{u,v}=\pr _{1,2}(\tilde s^{-1}(W_{d+u}^{0,v}))\; .
  $$

Since $\tilde s$ is continuous and $\pr _{1,2}$ is proper,
  $$
  \bar W_d^{u,v}=
  \pr _{1,2}(\tilde s^{-1}(\bar W_{d+u}^{0,v})\; .
  $$
Thus to prove  the theorem it is enough to show that $\bar W_d^{0,v}$ is contained in $W_d$.

Let $(Z_1,Z_2)$ be a closed point of $\bar W_d^{0,v}$. If $(Z_1,Z_2)$ is in $W_d^{0,v}$, then it is also in $W_d$. Suppose
  $$
  (Z_1,Z_2)\in \bar W_d^{0,v}\smallsetminus W_d^{0,v}\; .
  $$
Let $W$ be an irreducible component of the quasi-projective variety $W_d^{0,v}$ whose Zariski closure $\bar W$ contains the point $(Z_1,Z_2)$. Let $U$ be an affine neighbourhood of $(Z_1,Z_2)$ in $\bar W$. Since $(Z_1,Z_2)$ is in the closure of $W$, the set $U\cap W$ is non-empty.

Let us show that we can always take an irreducible curve $C$ passing through $(Z_1,Z_2)$ in $U$. Indeed, write $U$ as $\Spec (A)$. It is enough to show that there exists a prime ideal in $\Spec (A)$ of height $n-1$, where $n$ is the dimension of $\Spec (A)$, where $A$ is Noetherian. Since $A$ is of dimension $n$ there exists a chain of prime ideals
  $$
  \gop _0 \subset \gop _1 \subset \cdots \subset \gop_n=\gop
  $$
of maximum length. Now consider the subchain
   $$
   \gop _0 \subset \gop _1 \subset \cdots \subset \gop _{n-1}\; .
   $$
This is a chain of prime ideals and $\gop _{n-1}$ is a prime ideal of height $n-1$, so we get an irreducible curve.

Let $\bar C$ be the Zariski closure of $C$ in $\bar W$. Two evaluation regular morphisms $e_0$ and $e_{\infty }$ from $\Hom ^v(\PR ^1,C_d^p(X))$ to $C_d^p(X)$ give the regular morphism
  $$
  e_{0,\infty }:\Hom ^v(\PR ^1,C_d^p(X))\to C_{d,d}^p(X)\; .
  $$
Then $W_d^{0,v}$ is exactly the image of the regular morphism $e_{0,\infty }$, and we can choose a quasi-projective curve $T$ in $\Hom ^v(\PR ^1,C_d^p(X))$, such that the closure of the image $e_{0,\infty }(T)$ is $\bar C$.

Consider the curve $C$ in $W$,  it is contained in $W_d^{0,v}$. We know that the image of $e_{0,\infty }$ is $W_d^{0,v}$. Consider the inverse image of $\bar C$ under the morphism $e_{0,\infty }$. Since $\bar C$ is a curve, the dimension of $e_{0,\infty }^{-1}(C)$ is greater than or equal than $1$. So it contains a curve. Consider two points on $\bar C$, consider their inverse images under $e_{0,\infty }$. Since $\Hom ^v(\PR ^1,C^p_d(X))$ is a quasi  projective variety, $e_{0,\infty }^{-1}(\bar C)$ is also quasi-projective, we can embed it into some $\PR ^m$ and consider a smooth hyperplane section through the two points fixed above. Continuing this procedure we get a curve containing these two points and contained in $e_{0,\infty }^{-1}(C)$. Therefore we get a curve $T$ mapping onto $\bar C$. So the closure of the image of $T$ is $\bar C$.

Now, as we have mentioned above, $\Hom ^v(\PR ^1, C _d^p(X))$ is a quasi-projective variety. This is why we can embed it into some projective space $\PR ^m$. Let $\bar T$ be the closure of $T$ in $\PR ^m$ and $\tilde T$ be the normalization of $\bar T$. Let $\tilde T_0$ be the pre-image of $T$ in $\tilde T$. Consider the composition
  $$
  f_0:\tilde T_0\times \PR ^1\to T\times \PR ^1\subset
  \Hom ^v(\PR ^1,C_d^p(X))\times \PR ^1
  \stackrel{e}{\to }C_d^p(X)\; ,
  $$
where $e$ is the evaluation morphism $e_{\PR ^1,C_d^p(X)}$. The regular morphism $f_0$ defines a rational map
  $$
  f:\tilde T\times \PR ^1\dasharrow C_d^p(X)
  $$
Then by resolution of singularities we get that $f$ could be extended to a regular map from $(\tilde T\times \PR^1)'$ to $C_d^p(X)$, where $(\tilde T\times \PR^1)'$ denote the blow up of $\tilde T\times \PR^1$ along the indeterminacy locus which is a finite set of points. Continue to call the strict transform of $\tilde T $ in the blow up as $\tilde T$, and the pre-image of $T$ as $\tilde {T_0}$

%(combine Theorem 3 in Ch. II, $\S 3.1$ and Theorem 3 in Ch. IV, %$\S 3.3$ in \cite{Shafarevich}).
The regular morphism $\tilde T_0\to T\to \bar C$ extends to the regular morphism $\tilde T\to \bar C$. Let $P$ be a point in the fibre of this morphism at $(Z_1,Z_2)$. For any closed point $Q$ on $\PR ^1$ the restriction $f|_{\tilde T\times \{ Q\} }$ of the rational map $f$ onto $\tilde T\times \{ Q\} \simeq \tilde T$ is regular on the whole curve $\tilde T$, because $\tilde T$ is non-singular. Then
  $$
  (f|_{\tilde T\times \{ 0\} })(P)=Z_1
  \qquad \hbox{and}\qquad
  (f|_{\tilde T\times \{ \infty \} })(P)=Z_2\; .
  $$
$$f:\{P\}\times \PR^1\to C^p_d(X)$$
has the property that
$$f(0)=Z_1,\quad f(\infty)=Z_2\;.$$

Hence $W_d^{0,v}$ is Zariski closed and consequently $W_d^{u,v}$ is Zariski closed. Therefore ${\theta^p_d}^{-1}(z)$ is a countable union of Zariski closed subsets in the product $C^p_{d,d}(X)$.
\end{proof}
By  Theorem \ref{theorem1}, we can define the dimension of the fibers of $\theta^p_d$ to be the maximum of the dimension of the Zariski closed subsets occuring in ${\theta^p_d}^{-1}(z)$, for $z$ in $\CH^p(X)_{\alg}$. Now consider the  subset of $C^p_{d}(X)\times C^p_{d}(X)$ consisting of  points  such that  the dimension of ${\theta^p_d}^{-1}(\theta^p_d)(Z)$ is not  constant as $Z$ varies. By the existence of Hilbert schemes this is a countable union of Zariski closed subsets of $C^p_{d}(X)\times C^p_d(X)$. Call this subset $B$. Then for $Z$ in the complement of $B$, the dimension of ${\theta^p_d}^{-1}(\theta^p_d)(Z)$ is constant and say $r$.

\begin{definition}
The dimension of the image of $\theta^p_d$ is defined to be equal to $2\dim(C^p_{d}(X))-r$.
\end{definition}
Suppose that there exists a codimension $p$ prime cycle on $X$ of degree $e$. Then this prime cycle gives rise to an embedding of $C^p_{d}(X)$ into $C^p_{d+e}(X)$. Hence we have
$$\dim(C^p_d(X))\leq \dim(C^p_{d+e}(X))\;.$$
Hence we can define the limit superior of the
$$\dim(\im(\theta^p_d))\;.$$
We say that $\CH^p(X)_{\alg}$ is infinite dimensional (under the assumption (*)) if
$$\limsup_d \dim(\im(\theta^p_d))=\infty$$
and finite dimensional otherwise.

\begin{theorem}
\label{theorem4}
Assume (*) and that there exists a linear codimension $p$ subspace (prime cycle of  degree $1$) of $X$.
Then the group $\CH^p(X)_{\alg}$ is representable if and only if it is finite dimensional.
\end{theorem}
\begin{proof}
The proof follows the approach of [\cite{Vo}, Proposition 10.10] for symmetric powers and zero cycles.
Suppose that $\CH^p(X)_{\alg}$ is representable. Then there exists $d$ such that $\theta^p_d$ is surjective. For every integer $n$ consider the subset
$$R\subset C^p_n(X)\times C^p_n(X)\times C^p_d(X)\times C^p_d(X)$$
consisting of quadruples
$$(Z_1,Z_2,Z_1',Z_2')$$
such that
$$\theta^p_n(Z_1,Z_2)=\theta^p_d(Z'_1,Z'_2)\;.$$
As  $\theta^p_d$ is surjective, it follows that the projection
$$\pr_1:R\to C^p_n(X)\times C^p_n(X)$$
is surjective. Now by Theorem \ref{theorem1}, $R$ is a countable union of Zariski closed subsets in the ambient variety $C^p_n(X)\times C^p_n(X)\times C^p_d(X)\times C^p_d(X)$. We prove it as a separate lemma:

\begin{lemma}
The set $R$ is a countable union of Zariski closed subsets in $C^p_n(X)\times C^p_n(X)\times C^p_d(X)\times C^p_d(X)$.
\end{lemma}

\begin{proof}
This follows from the above Theorem \ref{theorem1}, but we present a proof for the convenience of the reader.
Let $Z_1,Z_2$ be two codimension $p$ cycles. They are rationally equivalent if there exists a positive cycle $Z'$, such that $Z_1+Z',Z_2+Z'$ belong to $C^p_d(X)$ for some fixed $d$, and there exists a regular morphism $f$ from $\PR^1$ to $C^p_d(X)$, such that
$$f(0)=Z_1+Z',\quad f(\infty)=Z_2+Z'\;.$$

Let us consider two cycles $(Z_1,Z_2,Z_1',Z_2')$ belonging to $R$. That would mean that the cycle class of $Z_1-Z_2$ is rationally equivalent to that of $Z_1'-Z_2'$.. So there exists a positive cycle $Z$ and a regular map $f$ from $\PR^1$ to $C^p_{d+n+u}(X)$  such that
$$f(0)=Z_1+Z_2'+Z,\quad f(\infty)=Z_1'+Z_2+Z'\;.$$
So it is natural to consider the following subvarieties of $C^p_{d+n+u}(X)\times C^p_{d+n+u}(X)$ denoted by $W^{u,v}_{d+n}$. It is  given by the collection of all elements  in the image of $C^p_d(X)\times C^p_d(X)\times C^p_n(X)\times C^p_n(X)$ in $C^p_{d+n+u}(X)\times C^p_{d+n+u}(X)$, under the natural map, given by
$$(Z_1,Z_2,Z_1',Z_2')\mapsto (Z_1+Z_2',Z_1'+Z_2)$$
Call this image as $F$.

so that there exist $Z\in C^p_{u}(X)$ and  $f$ in $\Hom^v(\PR^1,C^p_{d+n+u}(X))$, for some positive integer $u$ satisfying
$$f(0)=Z_1+Z_2'+Z,\quad f(\infty)=Z_2+Z_1'+Z'\;.$$
Here $\Hom^v(\PR^1,C^p_{d+n+u}(X))$ is the Hom-scheme of degree $v$ morphisms from $\PR^1$ to $C^p_{d+n+u}(X)$.
For working purpose denote $\prod_{i=1}^n C^p_{d_i}(X)$ as $C^p_{d_1,\cdots,d_n}(X)$.

  Let
  $$
  e:\Hom ^v(\PR ^1,C^p_{d+n+u,d+n+u}(X))\to
  C^p_{d+n+u,d+n+u}(X)
  $$
be the evaluation morphism sending $f:\PR ^1\to C^p_{d+n+u,d+n+u}(X)$ to the ordered pair $(f(0),f(\infty ))$, and let us consider the diagonal in $C^p_{u,u}(X)$ and multiply with  $F$ and consider:
  $$
  s:F\times \Delta_{C^p_{u,u}(X)}\to C^p_{d+n+u,d+n+u}(X)
  $$
 the regular morphism sending $(Z_1+Z_2',Z_2+Z_1')$ to $(Z_1+Z_2'+Z,Z_2+Z_1'+Z')$. The two morphisms $e$ and $s$ allow to consider the fibred product
  $$
  V=\Hom ^v(\PR ^1,C^p_{d+n+u,d+n+u}(X))\times _{C^p_{d+n+u,d+n+u}(X)}(F\times \Delta_{C^p_{u,u}(X)})\; .
  $$
This $V$ is a closed subvariety in the product
  $$
  \Hom ^v(\PR ^1,C^p_{d+n+u,d+n+u}(X))\times F \times \Delta_{C^p_{u,u}(X)}
  $$
over $\Spec (k)$ consisting of tuples $(f,Z_1+Z_2',Z_2+Z_1',Z)$ such that
  $$
  e(f)=s(Z_1+Z_2',Z_2+Z_1',Z)\; ,
  $$
i.e.
  $$
  (f(0),f(\infty ))=(Z_1+Z_2'+Z,Z_2+Z_1'+Z)\; .
  $$
The latter equality gives
  $$
  V= W_{d+n}^{u,v}\; .
  $$

Vice versa, if $(Z_1+Z',Z_2+Z_1')$ is a closed point of $W_{d+n}^{u,v}$, there exists a regular morphism
  $$
  f\in \Hom ^v(\PR ^1,C^p_{d+n+u,d+n+u}(X))
  $$
with $ f(0)=Z_1+Z_2'+Z$ and $ f(\infty )=Z_2+Z_1'+Z$. Then $(f,Z_1+Z_2',Z_2+Z_1',Z)$ belongs to $V$.

So the set $W_{d+n}^{u,v}$ is itself a quasi-projective variety.

Suppose that $(Z_1+Z_2',Z_2+Z_1')$ is in $W_{d+n}^{u,v}$. Then there exists $f$ in $\Hom^v(\PR^1,C^p_{d+n+u}(X))$, $Z$ in $C^p_u(X)$ such that

$$f(0)=Z_1+Z_2'+Z,\quad f(\infty)=Z_2+Z_1'+Z\;.$$
This equality immediately imply that $(Z_1+Z_2'+Z,Z_2+Z_1'+Z)$ is in $W_{d+n+u}^{0,v}$. On the other hand, consider the map
$$\tilde{s}:F\times \Delta_{C^p_{u,u}(X)}\to C^p_{d+n+u,d+n+u}(X)$$
given by
$$(Z_1+Z_2',Z_2+Z_1',Z)\mapsto (Z_1+Z_2'+Z,Z_2+Z_1'+Z)\;.$$
By the above
$$W_{d+n}^{u,v}\subset \pr_{1,2}(\tilde{s}^{-1}(W_{d+n+u}^{0,v}))\;.$$
Conversely suppose that $(Z_1+Z_2',Z_1'+Z_2)$ belongs to $\pr_{1,2}(\tilde{s}^{-1}(W_{d+n+u}^{0,v}))$. Then there exists $f$ in $\Hom^v(\PR^1,C^p_{d+n+u}(X))$ satisfying
$$f(0)=Z_1+Z_2'+Z',\quad f(\infty)=Z_2+Z_1'+Z'\;.$$
This gives, $(Z_1+Z_2',Z_2+Z_1')$ belongs to $W_{d+n}^{u,v}$.
Hence
  $$
  W_{d+n}^{u,v}=\pr _{1,2}(\tilde s^{-1}(W_{d+n+u}^{0,v}))\; .
  $$

Since $\tilde s$ is continuous and $\pr _{1,2}$ is proper,
  $$
  \bar W_{d+n}^{u,v}=
  \pr _{1,2}(\tilde s^{-1}(\bar W_{d+n+u}^{0,v})\; .
  $$
So to prove the second assertion of the proposition it is enough to show that $\bar W_{d+n}^{0,v}$ is contained in $W_{d+n}$. This follows by arguing as in \ref{theorem1}.

\end{proof}

Let us write
$$R=\cup_i R_i$$
where each $R_i$ is a Zariski closed subset in $C^p_n(X)\times C^p_n(X)\times C^p_d(X)\times C^p_d(X)$. Considering the projection from $\pr_1$ from $R$ to $C^p_n(X)\times C^p_n(X)$, we have $\cup_i \pr_1(R_i)=C^p_n(X)\times C^p_n(X)$. But $C^p_n(X)\times C^p_n(X)$ can be uniquely decomposed into finitely many Zariski closed irreducible subsets of maximal dimension. Using the fact that the ground field $k$ is uncountable, it will follow that there exists finitely many components $R_1,\cdots,R_m$ of $R$, such that $\cup_i R_i=R'$ surjects onto $C^p_n(X)\times C^p_n(X)$. So we have that
$$\dim R'\geq 2\dim C^p_d(X)\;.$$
Now consider $(Z_1,Z_2,Z'_1,Z'_2)$ in $\cup R_i$, then we have
$$\dim_{(Z_1,Z_2)}R'\cap (C^p_n(X)\times C^p_n(X)\times (Z'_1,Z'_2))\geq 2\dim C^p_n(X)-2\dim C^p_d(X)\;.$$
This number on the right hand side is bigger than zero if we take sufficiently large $n$ such that $C^p_n(X)$ contains $C^p_d(X)$. Then the above is an algebraic set contained in
$${\theta^p_{n}}^{-1}(\theta^p_d(Z'_1,Z_2'))\;.$$
As $(Z_1,Z_2)$ is arbitrary and the projection
$$R'\to C^p_n(X)\times C^p_n(X)$$
is surjective, we have that dimension of ${\theta^p_n}^{-1}(Z_1,Z_2)$ is atleast $2(\dim (C^p_n(X))-\dim(C^p_d(X)))$. Hence the dimension of the image of $\theta^p_n$ is bounded by $2\dim C^p_d(X)$. Therefore $\CH^p(X)_{\alg}$ is finite dimensional.

\medskip

Now suppose that $\CH^p(X)_{\alg}$ is finite dimensional. We have to prove that there exists $d$ such that $\theta^p_d$ is surjective. Let $d$ be such that
$$\dim(\im(\theta^p_d))=\dim(\im(\theta^p_{d+e}))$$
for all positive integer $e\geq 1$. Let $V$ be a subvariety of degree $1$ and codimension $p$ (this exists by the assumption), giving an embedding of $C^p_{d}(X)\times C^p_d(X)$ into $C^p_{d+1}(X)\times C^p_{d+1}(X)$. Call  the embedding as $i_V$, then we have
$$\theta^p_{d+1}\circ i_V=\theta^p_d\;.$$
Let $F$ be the fiber of $\theta^p_d$ passing through a general point $(Z_1,Z_2)$ of $C^p_d(X)\times C^p_d(X)$, let $F'$ be the fiber of $\theta^p_{d+1}$ through a general point in $C^p_{d+1}(X)\times C^p_{d+1}(X)$. By assumption we have that
$$2\dim(C^p_d(X))-r_1=2\dim(C^p_{d+1}(X))-r_2$$
where $r_1,r_2$ are dimensions of $F,F'$. Thus
$$r_2-r_1=2n$$
where $n=\dim(C^p_{d+1}(X))-\dim(C^p_d(X))\;.$
Let $F''$ be the fiber of a $\theta^p_{d+1}$ such that it passes through $(Z_1+V,Z_2+V)=i_V(Z_1,Z_2)$.
Then by the definition of dimension of the fiber of $\theta^p_{d+1}$ we have that
$$\dim(F'')\geq \dim(F')\;.$$
Now consider the subset
$$R=\{(Z_1,Z_2,Z'_1,Z'_2):\theta^p_{d+1}(Z_1,Z_2)=\theta^p_d(Z'_1,Z'_2)\}\;.$$
The projection from $R$ to $C^p_d(X)\times C^p_d(X)$ is surjective, so there exists finitely many irreducible subsets containing $(Z_1+V,Z_2+V,Z_1,Z_2)$ such that there union $R'$ dominates  $C^p_d(X)\times C^p_d(X)$. Fiber of this projection from $R'$ is of dimension greater or equal than that of $F''$. So it is greater than or equal to $\dim(F)+2\dim C^p_d(X)$. The fibers of the first projection
$$p:R'\to C^p_{d+1}(X)\times C^p_{d+1}(X)$$
are of dimension at most $\dim(F)$. So we have
$$\dim p(R')\geq \dim R'-\dim F\geq \dim F''+2\dim (C^p_d(X))-\dim(F)$$
$$\geq \dim F'-\dim F+2\dim(C^p_d(X)=\dim (C^p_{d+1}(X)\times C^p_{d+1}(X))$$
Therefore $p$ is surjective from $R'$ to $C^p_{d+1}(X)\times C^p_{d+1}(X)$ and hence $\im(\theta^p_d)=\im(\theta^p_{d+1})$. By induction we have $$\im(\theta^p_d)=\im(\theta^p_{d+e})$$
for all positive integer $e\geq 1$. Also observe that
$$\cup_i \im(\theta^p_i)=\CH^p(X)_0\;.$$
This implies that $\theta^p_d$ is surjective onto $\CH^p(X)_0$. By the assumption of the theorem we have $\CH^p(X)_0$ is equal to $\CH^p(X)_{\alg}$. Therefore $\theta^p_d$ is surjective onto $\CH^p(X)_{\alg}$.
\end{proof}
Now our aim is to detect the kernel of the Abel-Jacobi map for higher dimensional cycles. Let's recall that the Abel-Jacobi map has domain ${\CH^p(X)}_{\hom}$ and target the Intermediate Jacobian $IJ^p(X)$ given by
$$H^{2p-1}(X,\CC)/F^p H^{2p-1}(X,\CC)\oplus H^{2p-1}(X,\ZZ)\;.$$
The first theorem in this direction is to relate the representablity of the Chow group of codimension $p$ (algebraically trivial) cycles with zero cycles on a smooth projective curve. Under the assumption as in Theorem \ref{theorem4} we have the following theorem:
\begin{theorem}
Suppose that there exists a smooth projective curve $C$, and a correspondence $\Gamma$ on $C\times X$ such that
$$\Gamma_*:\CH^1(C)_{\alg}\to \CH^p(X)_{\alg}$$
is surjective. Then $\CH^p(X)_{\alg}$ is representable.
\end{theorem}
\begin{proof}
The proof is along the same line as in [\cite{Vo}, Proposition 10.12].
Note that the natural map from $\Sym^g C\times \Sym^g C$ to $\CH^1(C)_{\alg}$ is surjective, where $g$ is the genus of the curve $C$. Therefore  $\CH^1(C)_{\alg}$ is finite dimensional. Therefore the image of $\Gamma_*$ is finite dimensional. But $\Gamma_*$ is surjective. So $\CH^p(X)_{\alg}$ is finite dimensional hence representable by Theorem \ref{theorem4}.
\end{proof}

It is difficult to prove the converse, that is : suppose $\CH^p(X)_{\alg}$ is representable in the sense mentioned above as in Definition \ref{definition1} then does there exist a curve $C$ and a correspondence $\Gamma$ on $C\times X$, such that $\Gamma_*$ is onto? Let us consider the following situation:

I) The Chow group of codimension $p$ cycles are generated by linear subspaces, that is the natural map from $\CH_0(F(X))$ to $\CH^p(X)$ is surjective. Here $F(X)$ is the Fano variety of linear subspaces of codimension $p$. Under this condition, the assumption (*) is satisfied.

\begin{theorem}
\label{theorem2}
Let $X$ be as above. Suppose that $\CH^p(X)_{\alg}$ is representable. Then there exists a smooth projective curve $C$ and a correspondence $\Gamma$ on $C\times X$, such that
$$\Gamma_*:\CH^1(C)_{\alg}\to \CH^p(X)_{\alg}$$
is surjective.
\end{theorem}
\begin{proof}
The proof of this theorem follows the approach of the proof of [\cite{Vo}, Proposition 10.12].
Consider the map $\theta^p_d$ from $C^p_{d}(X)$ to $\CH^p(X)_{\alg}$ given by
$$Z\mapsto Z-dL_0$$
where $L_0$ is a fixed linear subspace of $X$. Since $\CH^p(X)_{\alg}$ is actually generated by linear subspaces, the above map restricted to $\Sym^d F(X)$ is surjective, continue to call it $\theta^p_d$. Let us consider large $d$ such that $\dim(\im(\theta^p_d))$ is constant and equal to $K$. Then the dimension of a general fiber is equal to
$$\dim(\Sym^d F(X))-K\;.$$

Now we prove that an irreducible component $Z$ of maximal dimension of a general fiber of $\theta^p_d$ cannot be contained in a set of the form
$$\Sym^{d-i}F(X)+ W$$
where $W$ is in $\Sym^i F(X)$, $\dim W<i$ and the above $+$ means the image of the natural map from
$$\Sym^{d-i}F(X)\times W\to \Sym^{d}F(X)\;.$$
If possible, assume that $Z$ is contained in such a set. Note that the dimension of $Z$ is $\dim(\Sym^{d}F(X))-K$. So we have
$$\dim(\Sym^{d-i}F(X))\geq \dim(\Sym^d F(X))-K-i+1\;.$$
Let the dimension of $F(X)$ be $n$. Then the above says
$$n(d-i)\geq nd-K-i+1$$
which implies
$$i<K/(n-1)\;.$$
Consider the subset
$$Z'=\{(z,w)|z+w\in Z\}\subset \Sym^{d-i}F(X)\times W\;.$$
By definition this set dominates $Z$, hence is of dimension greater or equal than $nd-K$. So the general fibers of the second projection $\pr_2:Z'\to W$ are of dimension at least
$$nd-K-i+1\;.$$
Also note that
$$\theta^p_{d}(z+w)=\theta^p_{d-i}(z)+\theta^p_{i}(w)\;.$$
Since $\theta^p_d$ is constant along $Z$, we have $\theta^p_{d-i}$ is constant along $Z'_w$. Thus if $Z$ passes through a very general point of $\Sym^d F(X)$, then $Z'_w$ passes through a very general point of $\Sym^{d-i}F(X)$. So we have $\dim(Z'_w)$ is less than the dimension of a general fiber of $\theta^p_{d-i}$ for generic $w$.  Now the dimension of $Z'_w$ is greater than or equal to
$$nd-K-i+1$$
but the dimension of the fiber of $\theta^p_{d-i}$ is equal to $(d-i)n-K$ because $d-i>d-K/n-1$ can be chosen to be arbitrarily large.

Let us assume that $d\geq 2$, and we have $nd-K\geq d$. Consider the following lemma.

\begin{lemma}
\label{lemma1}
Let $Y$ be an ample hypersurface of $F(X)$ and let $Z$ be an irreducible subset of $\Sym^d F(X)$ not contained in any subset of the form
$$\Sym^{d-i}F(X)+W$$
with $W\subset \Sym^i F(X)$ and dimension of $W$ is less than $i$. Then $Z$ intersects $\Sym^d Y$, provided that $\dim(Z)\geq d$.
\end{lemma}
Therefore by applying the lemma we see that a general fiber of $\theta^p_d$ intersects $\Sym^d Y$, for sufficiently large $d$ and provided that $n\geq 2$. Hence $\theta^p_d$ and $\theta^p_d|_{\Sym^d Y}$ have same image and the later has image of bounded dimension. So we can apply the lemma again and finally get that $\theta^p_d$ and $\theta^p_d|_{\Sym^d C}$ have same image, where $C$ is a smooth projective curve obtained by intersecting $n-1$ many ample hypersurfaces. This proves the theorem.
\end{proof}

\smallskip

\textit{Proof of Lemma \ref{lemma1}}:

\begin{proof} The proof is along the line of [\cite{Vo}, Lemma 10.13]. Consider the quotient map $r:F(X)^d\to \Sym^d F(X)$. Let $r^{-1}(Z)=\wt{Z}$, let $\wt{Z_0}$ be a component of $\wt{Z}$ dominating $Z$. By the hypothesis we have the following:

for every $i\geq 1$ and every subset $I$ of cardinality $i$, we have $\dim p_I(\wt{Z_0})\geq i$, where $p_I$ is the projection from $F(X)^d$ to $F(X)^i$ corresponding to the set of indices. Since $\wt{Z_0}$ dominates $Z$, it is sufficient to prove that $\wt{Z_0}$ intersects $Y^d$, for an ample hypersurface $Y $  in $F(X)$. Consider a complete intersection $V$ in $\wt{Z_0}$, which is obtained by intersecting $\wt{Z_0}$ with finitely many ample hypersurfaces. So that dimension of $V$ is $d$. Then the hypotheses on $\wt{Z_0}$ implies that same would be true for  $V$. So without loss of generality we can assume that dimension of $\wt{Z_0}=d$. Consider a de-singularization $Z'$ of $\wt{Z_0}$. Consider the divisors
$$D_i:=(\pr_i\circ \tau)^{-1}(Y)$$
where $\tau$ is the natural map from $Z'$ to $\wt{Z_0}$. Now $Y$ is ample, so we have
$$(\pr_i\circ \tau)_*((\pr_i\circ\tau)^*(Y).C)=Y.(\pr_i\tau)_*C\geq 0$$
which means that $D_i$'s are numerically effective. Our claim will follow from the fact that $D_1\cap\cdots\cap D_d$ is non-empty. So we prove that $D_1\cap\cdots\cap D_d$ is non-empty. First suppose that $d=2$. We have $D_1,D_2$ two divisors numerically effective. Hence we have
$$D_1^2\geq 0, \quad D_2^2\geq 0$$
Suppose that $D_1.D_2=0$. Then the intersection matrix of $(D_1,D_2)$ is semipositive. So by the Hodge index theorem we have $D_1=rD_2$ for some integer $r$. Hence $(D_1+D_2)^2=0$.  But $D_1+D_2$ is the pull-back of an ample divisor on $F(X)\times F(X)$, under a generically finite map. So it is ample. Therefore $(D_1+D_2)^2>0$, which is a contradiction. The general case follows from this by induction on the dimension of $\wt{Z_0}.$
Consider the divisor ${D_1}|_{D_2\cap \cdots\cap D_{d-1}}$ and ${D_d}|_{D_2\cap\cdots\cap D_{d-1}}$ embedded in $D_2\cap\cdots\cap D_{d-1}$.
We know that $D_1.(D_1\cap D_2\cdots\cap D_{d-1})\geq 0$, which means,
$$(D_1)^2|_{D_2\cap\cdots\cap D_{d-1}}\geq 0$$
here we use the ampleness of the divisor $D_1$.
Similarly
$$(D_d)^2|_{D_2\cap\cdots\cap D_{d-1}}\geq 0\;.$$
Now suppose that $D_1\cap\cdots\cap D_d=\emptyset$. Again as above, applying Hodge index theorem we have that
$$(D_1+D_d)^2|_{D_2\cap\cdots\cap D_{d-1}}=0\;.$$
Then it follows that
$${\tau_*(D_1+D_d)^2}|_{\tau(D_2\cap\cdots D_{d-1})}=0\;.$$
Now consider the projection $\pr_{1,d}:F(X)^d\to F(X)\times F(X)$ to the $1$-st and the $d$-th coordinate. Observe that  $D_1+D_2$ is the pull-back of an ample line bundle on $F(X)\times F(X)$ via a generically finite map.
The above gives
$${\pr_{1,d}}_*((\pr^*_1(Y_1)+\pr_d^*(Y_1))^2)|_{\tau(D_2\cap\cdots\cap D_{d-1})}=0\;.$$
But $${\pr_{1,d}}_*(\pr^*_1(Y_1)+\pr_d^*(Y_1))$$
is the divisor obtained by pulling back $Y_1$ to $F(X)\times F(X)$ by the two projections and adding them. This is ample on $F(X)\times F(X)$, so its cohomology class is non-trivial. Note that there exists $D'_1,\cdots,D'_{d-1}$ respectively in the linear systems of $D_i$ for $i=2,\cdots,d-1$ such that the image of $\tau(D'_2\cap\cdots\cap D'_{d-1})$ by $\pr_{1,d}$ is of dimension greater or equal than two. This is because there exists a generic divisor $D'_{d-1}$ in the linear system of $D_{d-1}$  such that its image under $Z'\to F(X)^d\to \Sym^d F(X)$ is not contained in a subset of the form
$$\Sym^{d-1-i}Y_1+W$$
such that $i\geq 1$ and $\dim W<i$, otherwise the condition on $Z$ is violated. Therefore we have
$$p_I(\tau(D'_{d-1}))\geq i$$
where $i$ is the cardinality of the indexing set $I$, which is a subset of $\{1,\cdots,d-1\}$. Continuing this process we get that there exists $D'_2,\cdots,D'_{d-1}$, such that
$$p_I(\tau(D'_2\cap\cdots \cap D'_{d-1}))\geq 2$$ for an indexing set $I$ of cardinality less than or equal to two.
Then by ampleness
$${\pr_{1,d}}_*((\pr^*_1(Y_1)+\pr_d^*(Y_1))^2)|_{\tau(D'_2\cap\cdots\cap D'_{d-1})})$$
cannot be zero as $\dim(\tau(D'_2\cap\cdots\cap D'_{d-1}))\geq 2$.
\end{proof}

So when, $\CH^p(X)_{\alg}$ is representable and $X$ satisfies the assumption of Theorem \ref{theorem2}, then the above Theorem \ref{theorem2} and arguments present in \cite{Vo} give that there exists an abelian variety $A$ and a correspondence $\Gamma$ supported on $A\times F(X)$, such that $L_*\Gamma_*:A\to \CH_0(F(X))_{\alg}\to\CH^p(X)_{\alg}$ is surjective, where $L_*$ is the universal incidence correspondence given by
$$\{(x,L):x\in L\}\;.$$

This leads us to the following result.

\begin{theorem}
\label{theorem3}
Let $X$ be smooth projective and it satisfies the hypotheses of Theorem \ref{theorem2}. Suppose that $\CH^2(X)_{\alg}$ is representable. Then the kernel of the Abel-Jacobi map is torsion.
\end{theorem}

\begin{proof}
The proof of this theorem follows the approach of [\cite{Vo}, Theorem 10.11]. By Theorem \ref{theorem2} there exists a smooth projective curve $C$ in $F(X)$ and a correspondence $\Gamma$ on $C\times F(X)$ such that $L_*\Gamma_*:J(C)\to \CH^2(X)_{\alg}$ is surjective. This yields further, a correspondence on $J(C)\times F(X)$, such that
$$L_*\Gamma_*:J(C)\to \CH^2(X)_{\alg}$$
is surjective. By Theorem \ref{theorem1} we have that kernel of $L_*\Gamma_*$ is a countable union of Zariski closed subsets of $J(C)$. Since kernel of $L_*\Gamma_*$ is a subgroup of $J(C)$ and we work over an uncountable ground field, the kernel is a countable union of translates of an abelian variety $A$ sitting in $\ker(L_*\Gamma_*)$. Now consider the supplementary abelian variety $B$, such that $A\times B\to J(C)$ is an isogeny.  Now replacing $J(C)$ by $B$, and $L\circ\Gamma$ by $(L\circ\Gamma)_{B\times X}$, we get that the kernel of $L_*\Gamma_*$ is countable.

Fix $x_0$ in $F(X)$, consider the subset $R$ of $F(X)\times B$ given by
$$\{(x,a):L_*\Gamma_*(a)=L_*(x)-L_*(x_0)\}$$
By Theorem \ref{theorem1} we have $R$ is a countable union of Zariski closed subsets in $F(X)\times B$. Since $L_*\Gamma_*$ is surjective, the projection from $R$ onto $F(X)$ is onto. Hence there exists a component $R_0$ of $R$ surjecting onto $F(X)$. Since $\ker(L_* \Gamma_*) $ is countable, the map is actually finite of say degree $r$. Thus $R_0$ gives rise to a correspondence of dimension equal to $\dim(F(X))$ between $F(X)$ and $B$, this provides a morphism
$$\alpha:F(X)\to B$$
given by
$$\alpha(x)=\alb_B(R_{0*}(x-x_0))\;.$$
By definition of $R_0$ we have that
$$L_*\Gamma_*\alpha(x)=r(L_*(x)-L_*(x_0))\;.$$
Now this $\alpha$ gives rise to a regular homomorphism from $\CH^2(X)_{\alg}$ to $B$. Hence by the universality of the intermediate Jacobian $IJ^2(X)$, there exists a unique regular map $\beta:IJ^2(X)\to B$ (for reference please see [\cite{K}, Theorem 1], [\cite{Mu}, Theorem 1.9]) such that
$$\alpha=\beta\circ \Phi_2$$
where $\Phi_2$ is the Abel-Jacobi map.
So we have
$$L_*\Gamma_*\beta\Phi_2(z)=rz\;.$$
This proves that kernel of $\Phi_2$ is torsion.
\end{proof}

\subsection{Application}

Consider $X$ to be a smooth cubic fourfold embedded in $\PR^5$. Then we know that the group of algebraically trivial one cycles $\CH^3(X)_{\alg}$ is generated by lines on $X$ (for a proof of this fact we refer to [\cite{Pa}, Section 4] and [\cite{Sh}, Theorem 1.1]). So the criterion for Theorem \ref{theorem2} is satisfied. We know, by \cite{Sc}, that there does not exist a smooth projective curve $C$ and a correspondence $\Gamma$ on $C\times X$, such that $\Gamma_*$ is surjective. Hence it follows that the natural map from the Chow varieties parametrising one cycles on $X$ does not surject onto $\CH^3(X)_{\alg}$.

\begin{theorem}
The group $\CH^3(X)_{\alg}$ is not representable. That is, the natural map from the product $C^3_d(X)\times C^3_d(X)$ to $\CH^3(X)_{\alg}$ is not surjective for any $d$.
\end{theorem}

The above Theorem \ref{theorem3} gives us a criterion by which we can detect the representability of $\CH^2(X)_{\alg}$, when $X$ satisfies the hypotheses of Theorem \ref{theorem2}. Precisely when, the Abel-Jacobi kernel is non-trivial, $\CH^2(X)_{\alg}$ is not representable, in the sense that the natural map from the Chow varieties (infact from the symmetric powers of $F(X)$) to $\CH^2(X)_{\alg}$ are not surjective. There are some examples of higher dimensional varieties with non-trivial Abel-Jacobi mappings given in \cite{GG}. Thus  Theorem \ref{theorem3} forces that such varieties cannot have, $\Sym^d F(X)$ surjecting onto $\CH^2(X)_{\alg}$. So for a threefold $X$, the non-triviality of the Abel-Jacobi kernel implies that the $\CH^2(X)_{\alg}$ cannot be generated by integral linear combination of lines on the threefold of any fixed degree.


\begin{thebibliography}{AAAAA}



\bibitem[Ba]{Ba} K.Banerjee, {\em One dimensional algebraic cycles on non-singular cubic fourfolds in $\PR^5$}, PhD Thesis, University of Liverpool, 2014.

\bibitem[BG]{BG} K. Banerjee and V. Guletskii, {\em Etale monodromy and rational equivalence of $1$-cycles on cubic hypersurfaces in $\mathbb P^5$}, Sbornik Math. Volume 211, no. 2, February 2020, 161-200.
\bibitem[B]{B} R.Barlow, {\em Rational equivalence of zero cycles for some surfaces with $p_g=0$}, Invent. Math. 1985, no. 2, 303-308.


\bibitem[BKL]{BKL}S.Bloch, A.Kas, D.Lieberman, {\em Zero cycles on surfaces with $p_g=0$}, Compositio Mathematica, tome 33, no. 2 (1976), page 135-145.

\bibitem[Fu]{Fulton} W. Fulton, {\em Intersection theory}, Ergebnisse der Mathematik und ihrer Grenzgebiete (3),  2.
      Springer-Verlag, Berlin, 1984.

\bibitem[GG]{GG}V.Guletskii, S.Gorchinsky, {\em Non-trivial elements in the Abel-Jacobi kernels of higher dimensional varieties}, Advances in Mathematics, 241, 2013.
\bibitem[GP]{GP} V.Guletskii, C.Pedrini, {\em Chow motive of a Goeadux surface}, Algebraic Geometry: A Volume in Memory of Paolo Francia, 179-195, de Gruyter, Berlin - New York, 2002.
%\bibitem[Gul]{Gul} V.Guletskii, {\em Motivic obstruction to rationality }, arXiv:1605.09434
%\bibitem[Gul1]{Gul1}V.Guletskii, {Bloch's conjecture for surfaces with involutions and of geometric genus zero}, arXiv:1704.04187
%\bibitem[Gul2]{Gul2} V.Guletskii, {Bloch's conjecture for the surface of Craighero and Gattazzo }, arXiv:1609.04074


\bibitem[IM]{IM} H.Inose, M.Mizukami, {\em Rational equivalence of 0-cycles on some surfaces with $p_g=0$}, Math. Annalen, 244, 1979, no. 3, 205-217.


\bibitem[K]{K} B.Kahn, {\em On the universal regular homomorphism in codimension $2$}, arxiv.org: 2007.07592.
\bibitem[M]{M} D.Mumford, {\em Rational equivalence for $0$-cycles on surfaces.}, J.Math Kyoto Univ. 9, 1968, 195-204.
\bibitem[Mu]{Mu} J.Murre, {\em Applications of algebraic K-theory to the theory of algebraic cycles}, Algebraic Geometry, Sitges(Barcelona), 1983, 216-261, Lecture Notes in Mathematics, 1124, Springer, Berlin, 1985.
\bibitem[Pa]{Pa} K. Paranjape, {\em Cohomological and cycle-theoretic connectivity}, Ann. of Math. (2) 139 (1994), no. 3, 641--660.


\bibitem[PW]{PW} C.Pedrini, C.Weibel, {\em Some examples of surfaces of general type for which Bloch's conjecture holds}, Recent advances in Hodge theory, Period Domains, Algebraic Cycles and Arithmetic, Edited by Matt Kerr and Gregory Pearlstein, Cambridge University Press, February 2016.


\bibitem[R]{R} A.Roitman, {\em $\Gamma$-equivalence of zero dimensional cycles (Russian)}, Math. Sbornik. 86(128), 1971, 557-570.
\bibitem[R1]{R1}A.Roitman, {\em Rational equivalence of 0-cycles}, Math USSR Sbornik, 18, 1972, 571-588
\bibitem[R2]{R2} A.Roitman, {\em The torsion of the group of 0-cycles modulo rational equivalence}, Ann. of Math. (2), 111, 1980, no. 3, 553-569
\bibitem[Sh]{Sh} M.Shen, {\em On relations among $1$-cycles on cubic hypersurfaces}, Journal of Algebraic Geometry, 23, 2014, 539-569.
\bibitem[SV]{SV} A.Suslin, V.Voevodsky, {\em Relative cycles and Chow sheaves}, Cycles, transfers, motivic homology theories, 10-86, Annals of Math studies.


\bibitem[Sc]{Sc} C.Schoen, {\em On Hodge structures and non-representability of Chow groups}, Composition Math. 88, no.3, 1993, 285-316.
%\bibitem[Voi]{Voi}C.Voisin,{\em Symplectic invoultions of K$3$ surfaces act trivially on $CH_0$}, Documenta Mathematicae 17, 851-860,2012.
\bibitem[Vo]{Vo} C.Voisin, {\em Complex algebraic geometry and Hodge theory II}, Cambridge studies of Mathematics, 2002.
\bibitem[V]{V}C.Voisin, {\em Bloch's conjecture for Catanese and Barlow surfaces}, Journal of Differential Geometry, 2014, no. 1, 149-175.
\bibitem[VC]{VC} C.Voisin, {\em Sur les zero cycles de certaines hypersurfaces munies d'un automorphisme}, Ann. Scuola Norm. Sup. Pisa Cl. Sci., (4), 19, 1992, no.4, 473-492.
%\bibitem[VC]{VC} C.Voisin, {\em Sur les zero cycles de certaines hypersurfaces munies d'un automorphisme}, Ann. Scuola Norm. Sup. Pisa Cl. Sci., (4), 19, 1992, no.4, 473-492.

\end{thebibliography}
\end{document}